\documentclass[10pt]{amsart}
\usepackage{amsfonts,amsmath,amssymb,amsthm,latexsym,verbatim,url,xcolor,nicematrix}
\newtheorem{theorem}{Theorem}[section]
\newtheorem{lemma}[theorem]{Lemma}

\newtheorem{corollary}[theorem]{Corollary}
\theoremstyle{remark}

\theoremstyle{definition}

\newcommand{\abs}[1]{|#1|}	
 
\newcommand{\R}{{\mathbb {R}}} 
\newcommand{\Z}{{\mathbb {Z}}} 
\newcommand{\Q}{{\mathbb {Q}}} 
\newcommand{\C}{{\mathbb {C}}}
\newcommand{\om}{{{\omega}}}

\newcommand{\gr}[1]{{\color{gray!75!black} #1}}
\newcommand{\ignore}[1]{{}}

\begin{document}
\title[]{Pretty good state transfer among large sets of vertices}

\author{Ada Chan}
\address{Department of Mathematics and Statistics\\York University\\ Toronto\\ ON\\ Canada N3J 1P3}
\author{Peter Sin}
\address{Department of Mathematics\\University of Florida\\ P. O. Box 118105\\ Gainesville FL 32611\\ USA}

\thanks{This work was partially supported by a grant from the Simons Foundation (\#633214
 to Peter Sin)}
\thanks{Ada Chan gratefully acknowledges the support of the Natural Sciences and Engineering Council of Canada (NSERC) Grant No. RGPIN-2021-03609}

\begin{abstract} In a continuous-time quantum walk on a network of qubits,
  pretty good state transfer is the phenomenon of state transfer between two vertices with fidelity arbitrarily close to 1.
  We construct families of graphs to demonstrate that there is no bound on the size of
  a set of vertices that admit pretty good state transfer between any two vertices of the set.  
\end{abstract}

\maketitle

\section{Introduction} 

Let $X$ be a simple finite graph, with adjacency matrix $A$, for some
fixed ordering of the vertex set $V(X)$.
Let $\C^{V(X)}$ denote the Hilbert space  in which the characteristic vectors
$e_a$, $a\in V(X)$, form an orthonormal basis.
A continuous-time quantum walk on $X$, based on the $XX$-Hamiltonian \cite[IV.E]{K}, is given by
the family $U(t)=\exp(-itA)$ of unitary matrices, for $t\in\R$, operating 
on $\C^{V(X)}$.
Two phenomena of central importance in the theory are {\it perfect
  state transfer} (PST) and {\it pretty good state transfer} (PGST).
Let $a$ and $b$ be vertices of $X$. We say that we have perfect state transfer from $a$ to $b$
at time $\tau$ if $\abs{U(\tau)_{b,a}}=1$. In other words, an initial state $e_a$
concentrated on the vertex $a$ evolves at time $\tau$ to one concentrated on $b$.
The concept of pretty good state transfer is an approximate version of perfect state transfer. We say that we have
pretty good state transfer from $a$ to $b$ if for every real number $\epsilon>0$  there exists
a $\tau\in\R$ such that $\abs{U(\tau)_{b,a}}>1-\epsilon$.
It is not hard to see that the relation
on $V(X)$, whereby vertices $a$ and $b$
are related if we have perfect state transfer from $a$ to $b$ at some time,
is an equivalence relation. Likewise
the relation defined by pretty good state transfer from $a$ to $b$ is
another equivalence relation .
 In these terms, a well known observation of Kay \cite[IV.D]{K} states that each PST-equivalence
class can have at most two members. This property of perfect state transfer is undesirable for the purpose of routing.
By contrast, \cite[example 4.1]{PB} shows that
in the cartesian product $P_2\square P_3$ of paths of lengths 2 and 3, the
4 vertices of smallest degree form a PGST-equivalence class. A natural question then
is: How large can a PGST-equivalence class be? The main aim of this paper is to
construct examples of graphs with arbitrarily large PGST-equivalence classes.
Our examples are generalizations of the example above in that they are 
finite cartesian products of paths of different lengths and the large PGST-equivalence class
is the set of vertices of minimum degree (the ``corners''). In order for the corners to be
PGST-equivalent the path lengths of the cartesian factors
have to be selected rather carefully to satisfy certain
arithmetic conditions which enter into the problem via Kronecker's approximation theorem,
a theorem whose relevance to pretty good state transfer was first noted in \cite{GKSS} and \cite{VZ}.

Since pretty good state transfer in a cartesian product implies pretty good state transfer for each cartesian factor, the paths involved must have pretty good state transfer between their end-vertices. Such paths have been classified in \cite{GKSS}; they are the paths of length $p-1$ or $2p-1$, where p is a prime, or of length $2^e-1$. 
Thus, we consider only cartesian products of paths of these lengths.
In Section~\ref{sc} we first classify the cartesian products
for which all corners are strongly cospectral,
as strong cospectrality is necessary for pretty good state transfer.
The results in Sections~\ref{nopgstsection} and \ref{pgstsection} lead to the following classification which is the main result of this paper.
\begin{theorem}\label{classification}
Let $X$ be a cartesian product of paths.  All corners of $X$ belong to the same PGST-equivalence class if and only if one of the following holds.
\begin{enumerate}
\item
Up to permutation of the cartesian factors, $X=P_{2^e-1}\square P_{p-1}$, for some $e\geq 2$ and prime $p\geq 3$.
\item
Up to permutation of the cartesian factors, $X=P_{2^e-1}\square P_{2p-1}$, for some $e\geq 2$ and prime $p\geq 3$.
\item
Up to permutation of the cartesian factors, 
$X= P_{p_1-1}\square \cdots\square P_{p_h-1}\square P_{2q_1-1}\square P_{2q_k-1}$,
where $h,k \geq 0$ and  $p_1,\dots,p_h, q_1, \dots, q_k$ are distinct primes such that 
\begin{equation*}
    p_1, \dots, p_h \equiv 1 \pmod 8
    \quad \text{and} \quad q_1,\dots,q_k \equiv 1 \pmod 4.
\end{equation*}
\item
Up to permutation of the cartesian factors, 
\begin{equation*}
    X=P_{2^e-1}\square P_{p_1-1}\square \cdots\square P_{p_h-1}\square P_{2q_1-1}\square P_{2q_k-1},
\end{equation*}  
where $h,k \geq 0$,  $e\geq 2$ and $p_1,\dots,p_h, q_1, \dots, q_k$ are distinct primes such that 
\begin{equation*}
    p_1, \dots, p_h \equiv 1 \pmod 8
    \quad \text{and} \quad q_1,\dots,q_k \equiv 1 \pmod 4.
\end{equation*}
\end{enumerate}
\end{theorem}

In Section~\ref{Lpgstsection}, we show on the contrary that there is no cartesian product of paths with pretty good state transfer among all corners when the Laplacian matrix is used as the Hamiltonian of the quantum walk.
\section{Notation and background results}
Let $A$ be the adjacency matrix of $X$.  We consider the spectral decomposition of $A$,
\begin{equation}
  A=\sum_{r=1}^k \theta_rE_r,
\end{equation}
where $\theta_1$,\dots, $\theta_k$ are the distinct eigenvalues of $A$ and $E_r$
is the idempotent projector onto the $\theta_r$ eigenspace.

Two vertices $a$ and $b$ are said to be {\it strongly cospectral} if and only if
for all $r$ we have $E_re_a=\pm E_re_b$.
The terminology is justified  by the fact that the above condition implies that
$(E_r)_{a,a}=(E_r)_{b,b}$ for all $r$, which is one of several equivalent definitions of {\it cospectrality} of $a$ and $b$.

Strong cospectrality is a fundamental notion in the study of quantum state transfer, and is a necessary condition for both
perfect state transfer and pretty good state transfer \cite{G12}.
The {\it eigenvalue support} $\Phi_u$ of a vertex $u$ is the set of eigenvalues $\theta_r$ for which $E_re_u\neq 0$. 
If $u$ and $v$ are strongly cospectral then $\Phi_u=\Phi_v$ and this set is the disjoint union
of $\Phi_{u,v}^+=\{\theta_r\mid E_re_u=E_re_v\}$ and $\Phi_{u,v}^-=\{\theta_r\mid E_re_u=-E_re_v\}$. 

If an eigenvalue of $X$ is simple, then the corresponding projector is a rank one symmetric matrix.
The following lemma is then immediate from the above definition of cospectrality.
\begin{lemma}\label{simplesc}
    If $X$ has simple eigenvalues then two vertices that are cospectral are
    strongly cospectral. \qed
  \end{lemma}

  The following theorem  \cite[Theorem 2]{BCGS} (also \cite[Lemma 2.2]{KLY}) is our main tool. It is a direct application of
  of Kronecker's approximation theorem to quantum walks.
  
  \begin{theorem}\label{PGSTcriterion} Let $X$ be a simple graph. Then two vertices
    $u$ and $v$ are in the same PGST-equivalence class if and only if the following conditions hold.
   \begin{enumerate}
   \item[(a)] $u$ and $v$ are strongly cospectral.
   \item[(b)] There is no sequence of integers $\{\ell_i\}$ such that all three
     of the following equations hold:
     \begin{enumerate}
     \item[(i)] $\sum_i\ell_i\theta_i=0$;
     \item[(ii)] $\sum_i\ell_i=0$; 
     \item[(iii)] $\sum_{i: \theta_i\in\Phi_{u,v}^-}\ell_i\equiv 1\pmod2$.
     \end{enumerate}
   \end{enumerate}\qed
 \end{theorem}

Let $P_n$ denote the path of length $n$.  Pretty good state transfer between extremal vertices has been characterized in \cite{GKSS}, and between internal vertices in \cite{vB}. We shall make use of
 the extremal case, in which pretty good state transfer occurs if and only if $n+1=p$ or $2p$, where $p$ is a prime, or if $n+1$ is a power of $2$.
We shall consider the cartesian product
\begin{equation}
    X=P_{n_1}\square P_{n_2}\square\cdots\square P_{n_k}
 \end{equation}
of $k$ paths, where $k$ is a positive integer.

\begin{lemma}\label{Lemma-1factor}
If pretty good state transfer occurs between $(x_1, y_1)$ and $(x_2, y_2)$ in $X\square Y$ then pretty good state transfer occurs between $x_1$ and $x_2$ in $X$.
\end{lemma}
\begin{proof}
Let $U_X(t)$ and $U_Y(t)$ denote the transition matrices of $X$ and $Y$, respectively.  Their cartesian product, $X\square Y$ has transition matrix $U_X(t) \otimes U_Y(t)$.
As
\begin{equation*}
\left\vert U_X(t)_{x_1,x_2}\right\vert,  \left\vert U_Y(t)_{y_1,y_2}\right\vert \geq \left\vert\big(U_X(t) \otimes U_Y(t)\big)_{(x_1,y_1), (y_1,y_2)}\right\vert,
\end{equation*}
pretty good state transfer from $(x_1, y_1)$ to $(x_2, y_2)$ in $X \square Y$ implies pretty good state transfer from $x_1$ to $x_2$ in $X$ and pretty good state transfer from $y_1$ to $y_2$ in $Y$.
\end{proof}

By a corner of  $X=P_{n_1}\square P_{n_2}\square\cdots\square P_{n_k}$, we shall mean a vertex $(a_1,\dots, a_k)$
in which every component $a_i$ is an end of $P_{n_i}$. There are $2^k$ corners.  If there is pretty good state transfer between any two corners of $X$, then
for $i=1,\dots, k$,
$n_i+1=p_i$ or $n_i+1=2p_i$, for some prime $p_i$, or  $n_i+1$ is a power of $2$.

\section{Large classes of  strongly cospectral vertices in path products.} \label{sc}
 It is well known (an unpublished result of G. Coutinho) that by taking cartesian products of paths of suitable lengths, one can obtain arbitrarily large equivalence classes of mutually  strongly cospectral vertices. In this section, we include a proof for completeness, and  in order to introduce notations that we shall need later on.

\begin{lemma} The automorphism group of $X$ acts transitively on the corners.
  Hence the corners are mutually cospectral. \qed
\end{lemma}

In order for the corners of $X=P_{n_1}\square P_{n_2}\square\cdots\square P_{n_k}$ to be mutually strongly cospectral,
we will need to choose the path lengths $n_i$ more carefully, so that $X$
will have simple eigenvalues, and Lemma~\ref{simplesc} will apply.
The proof of the simplicity
of the eigenvalues will make use of some well known properties the eigenvalues of paths
and of cyclotomic fields, which we shall now discuss.

The adjacency matrix of $P_n$ has eigenvalues $2\cos \frac{r\pi}{n+1}= e^{\frac{r\pi}{n+1}i} + e^{-\frac{r\pi}{n+1}i}$. 
The degree of $2\cos \frac{\pi}{n+1}$ is $d:=\frac{1}{2}\phi(2(n+1))$ where $\phi$ is the totient function.  
For $r=1,\dots, d$, using $T_r$ to denote the Chebyshev polynomial of the first kind of degree $r$, we have
\begin{equation*}
    \cos \frac{r\pi}{n+1}=T_r\left(\cos \frac{\pi}{n+1}\right).
\end{equation*} 
Thus
\begin{equation}\label{Eqn-basis}
    \left\{1, 2\cos \frac{\pi}{n+1}, 2\cos \frac{2\pi}{n+1}, \dots, 2\cos \frac{(d-1)\pi}{n+1}\right\}
\end{equation} 
is a basis of
the field, $F_{n+1}$, generated by the eigenvalues of $P_n$ over $\Q$.  

When $n+1=p$ for some prime $p$, let $\alpha_r:=2\cos \frac{r\pi}{p}$, ($r=1,\dots p-1$).  It follows from the minimal polynomial of the primitive $2p$-th root of unity that  
\begin{equation}\label{Eqn-p-alt-sum}
1+\sum_{j=1}^{\frac{p-1}{2}} (-1)^j\alpha_j=0.
\end{equation}

\begin{lemma}\label{galoisp} Let $p\geq 5$ be a prime and let  $\alpha_j=2\cos \left(\frac{j\pi}{p}\right)$, ($1\leq j\leq \frac{p-1}{2})$. Then for every $1\leq r < s\leq\frac{p-1}{2}$, the set 
    $\{1,\alpha_r,\alpha_s\}$ is linearly independent over $\Q$.
\end{lemma}
  \begin{proof}
The degree of $\alpha_1$ is $d=\frac{p-1}{2}$.
  If $1\leq r < s <d$, the set $\{1,\alpha_r,\alpha_s\}$ is a subset of the basis $\{1, \alpha_1,\dots, \alpha_{d-1}\}$ of $F_{p}$, so it is linearly independent.

It remains to consider $\alpha_d$. By Equation~(\ref{Eqn-p-alt-sum}),
    we may replace any element in 
    \begin{equation*}
    \left\{\alpha_j : 1\leq j \leq d-1 \text{ and } j \neq r\right\}
    \end{equation*}
    by $\alpha_d$ in $\{1, \alpha_1, \dots, \alpha_{d-1}\}$ to get another basis of $F_p$ containing $1$, $\alpha_r$ and $\alpha_d$.
  \end{proof}

The eigenvalues of $P_{2p-1}$ are $\beta_r:=2\cos \left(\frac{r\pi}{2p}\right)$, for $r=1,\dots, 2p-1$.
The field $F_{2p}$ generated by these eigenvalues over $\Q$
is the intersection of the $4p$-th cyclotomic field with the field of real numbers,
 so $\abs{F_{2p}:\Q}=p-1$. We also have $F_{2p}=\Q(\beta_1)$.
  The minimal polynomial of $e^{\frac{\pi}{2p}i}$ yields
\begin{equation}\label{Eqn-2p-alt-sum}
1+\sum_{j=1}^{\frac{p-1}{2}} (-1)^j\beta_{2j}=0.
\end{equation}

\begin{lemma}\label{galois2p} Let $p\geq 7$ be a prime and let  $\beta_r=2\cos \left(\frac{r\pi}{2p}\right)$, ($1\leq r\leq p-1)$. Then for every $1\leq r<s\leq p-1$, the set 
    $\{1,\beta_r,\beta_s\}$ is linearly independent over $\Q$.
\end{lemma}
  \begin{proof}
The degree of $\beta_1$ is $p-1$. For $1\leq r < s\leq p-2$, $\{1,\beta_r,\beta_s\}$ is a subset of the basis $\{1, \beta_1, \dots, \beta_{p-2}\}$.  Hence it is linearly independent over $\Q$.

Using Equation~(\ref{Eqn-2p-alt-sum}), we may replace any element in $\left\{\beta_{2j} : 1\leq j \leq \frac{p-3}{2} \text{ and } 2j \neq r\right\}$
     by $\beta_{p-1}$ in $\{1,\beta_1,\dots,\beta_{p-2}\}$ to get another basis of $F_{2p}$ containing $1$, $\beta_r$ and $\beta_{p-1}$.
  \end{proof}
  
\begin{corollary}\label{linindep} Suppose $n=p-1$ for some prime $p\geq 5$, or $n=2p-1$ for some prime $p\geq 7$. For any two non-zero 
  eigenvalues $\lambda$ and $\lambda'$ of $P_{n}$ such that $\lambda\neq\pm\lambda'$,
  the set $\{1,\lambda ,\lambda'\}$ is linearly independent over $\Q$.
\end{corollary}

The eigenvalues of $X$, counting multiplicity, are the $\prod_{i=1}^kn_i$ numbers
  $\lambda_1+\lambda_2+\cdots+\lambda_k$, where $\lambda_i$ is an eigenvalue of $P_{n_i}$.

\begin{lemma}\label{simple} 
Let $p_1,\dots, p_k \ge 5$ be distinct primes.  Suppose $n_i=p_i-1$ or $n_i=2p_i-1$. 
Then $X = P_{n_1}\square \cdots \square P_{n_k}$ has simple eigenvalues.
\end{lemma}
\begin{proof}
  We may assume $k>1$, since paths have simple eigenvalues.
  The eigenvalues of $X$ are the values $\lambda_1+\lambda_2+\cdots+\lambda_k$
  where $\lambda_i$ is an eigenvalue of $P_{n_i}$.
  Suppose
  \begin{equation}\label{equalsums}
    \lambda_1+\lambda_2+\cdots+\lambda_k=\lambda'_1+\lambda'_2+\cdots+\lambda'_k.
  \end{equation}
  where the $\lambda_i$ and $\lambda'_i$ are eigenvalues of $P_{n_i}$. Suppose for
  a contradiction that for some $i$ we have $\lambda_i\neq \lambda_i'$.  Then there must be another index $j$ with $\lambda_j\neq \lambda_j'$. Without loss of generality we can assume that $i=1$
  and that $p_1\ge 7$ if $n_1=2p_1-1$. 
    We rewrite  \eqref{equalsums} as
    \begin{equation}\label{equalsums2}
      ( \lambda_1'-\lambda_1)=(\lambda_2-\lambda'_2)\cdots+
      (\lambda_k-\lambda'_k).
    \end{equation}
    The left member of \eqref{equalsums2} lies in  the cyclotomic field of order $4p_1$ while
    the right member lies in a cyclotomic field of order $4m$, where $m$ is coprime to $p_1$.
    The intersection of these fields is $\Q(i)$, but we also know that the eigenvalues are real,
    so it follows that the common value of \eqref{equalsums2} must be rational.

    If, without loss of generality, $\lambda_1=0$ then $\lambda'_1 \notin \Q$ contradicting Equation~(\ref{equalsums2}). 
    Otherwise, as $\lambda_1\notin\Q$, we cannot have $\lambda_1'= - \lambda_1$.  Then  by Corollary~\ref{linindep}, the set $\{1,\lambda_1,\lambda_1'\}$
    is linearly independent over $\Q$, contradicting the rationality of $\lambda_1'-\lambda_1$ in \eqref{equalsums2}.  Hence $X$ has simple eigenvalues. 
\end{proof}

\begin{lemma}\label{simpleY} Let $X$ be as in Lemma~\ref{simple} and let
$Y=P_{2^e-1}\square X$, $e\ge 2$. Then the eigenvalues of $Y$ are
simple.
\end{lemma}
\begin{proof}
    The eigenvalues of $Y$, counting multiplicity are the $(2^e-1)\prod_{i=1}^k(n_i)$
complex numbers $\lambda_0+\lambda_i+\cdots+\lambda_k$,
where $\lambda_0$ is an eigenvalue of $P_{2^e-1}$ and for $1\leq i\leq k$
$\lambda_i$ is an eigenvalue of $P_{n_i}$.
Let  $\lambda'_i$, $i=0$,$1$,\dots, $k$ similarly denote eigenvalues and consider the equation
  \begin{equation}\label{equalevals}    \lambda_0+\lambda_1+\lambda_2+\cdots+\lambda_k=\lambda_0'+\lambda'_1+\lambda'_2+\cdots+\lambda'_k.
  \end{equation}
  We shall show that $\lambda_i=\lambda'_i$ for all i. 
  Suppose $\lambda_0=\lambda'_0$.
  Then we can cancel these terms and we have an equation expressing equality
  of two eigenvalues of $X$. So Lemma~\ref{simple} gives the desired conclusion. Therefore we may assume
  $\lambda_0\neq\lambda'_0$. 
 
  Let $F_{2^e}$
  denote the field generated by the eigenvalues of $P_{2^e-1}$. This is the
  intersection of the real field with
  the cyclotomic field $Q(\om)$,
  where $\om=e^{\frac{i\pi}{2^e}}$ is a
  primitive $2^{e+1}$-th root of unity.
Let $F_X$ denote the field generated 
by the eigenvalues of the path factors in $X$. Then $F_X$ lies in a cyclotomic 
field of order $4m$, where $m$ is odd.
It follows that $F_{2^e}\cap F_X=\Q$.
Thus, if we rearrange \eqref{equalevals} 
by isolating $\lambda_0-\lambda'_0$,
we see that $\lambda_0-\lambda'_0\in\Q$.
The eigenvalues of $P_{2^e-1}$ are
$\om^r+\om^{-r}$, for $r=1$, \dots, $2^e-1$. We assume for a contradiction that  $\lambda_0\neq\lambda'_0$.
The only rational eigenvalue is zero, for $r=2^{e-1}$, so $\lambda_0$ and 
$\lambda'_0$ must both be irrational.

  The Galois group of $Q(\om)$ over $\Q$
  consists of the $2^e$ automorphisms
  of the form $\om\mapsto \om^a$, where 
  $a\in \Z/2^{e+1}\Z$ is odd. 
  The Galois automorphisms 
 of $F_{2^e}=\Q(\om+\om^{-1})$ are obtained by restriction of those for $Q(\om)$, and form a cyclic group of order
 $2^{e-1}$.  By Galois theory,
 the subfields of $F_{2^e}$ correspond
  bijectively with the subgroups  of the Galois group. Thus, for each $d$ with
  $1\le d\leq e$ there is a unique
  subfield of degree $2^{d-1}$, and this subfield must be the field 
  $F_{2^d}=\Q(\om^{2^{e-d}})\cap\R$,
  as this field has the right degree.
  An eigenvalue $\om^r+\om^{-r}$ is
  Galois conjugate to $\om^{2^d}+\om^{-2^d}$, where $2^d$
  divides $r$ exactly, so 
  $Q(\om^r+\om^{-r})=F_{2^d}$.
  Suppose $\Q(\lambda_0)\neq\Q(\lambda'_0)$.
  Then $\{1,\lambda_0,\lambda'_0\}$
  is linearly independent over $\Q$,
  a contradiction. Therefore we may
  assume $\Q(\lambda_0)=\Q(\lambda'_0)$.
  Thus, $\lambda_0=\om^r+\om^{-r}$
  and $\lambda'_0=\om^s+\om^{-s}$, where $1\leq r,s\leq 2^e-1$
  and $r$ and $s$  are exactly divisible by the same power of 2,
  say $2^d$.  In other words $\om^r$ and $\om^s$ are both odd powers
  of $\om^{2^{e-d}}$, a primitive $2^{d+1}$-th root of unity.
  Now $2^d+1$ is odd and the $(2^d+1)$-th power of $\om^{2^{e-d}}$  is $-\om^{2^{e-d}}$.
  It follows that there is  Galois automorphism of $Q(\lambda_0)$
  sending $\lambda_0$ to $-\lambda_0$ and $\lambda'_0$ to $-\lambda'_0$.
  Since $\lambda_0-\lambda'_0$ is a nonzero rational number, we have
  our final contradiction. This proves that $Y$ has simple eigenvalues.
\end{proof}

From Lemmas~\ref{simple}, \ref{simpleY} and \ref{simplesc}, we draw
the following conclusion.
\begin{corollary}\label{Lemma-scprods}
Let $p_1,\dots, p_k \ge 5$ be distinct primes.  Suppose $n_i=p_i-1$ or $n_i=2p_i-1$. Further, let $n_0=2^e-1$, $e\ge2$. 
Then in $X = P_{n_1}\square \cdots \square P_{n_k}$ and $Y=P_{n_0}\square X$, all corners are mutually strongly cospectral.
\end{corollary}\qed

The following results imply that if the primes in the above corollary are not distinct, then the corners of $X=P_{n_1}\square P_{n_2}\square\cdots\square P_{n_k}$ are not mutually strongly cospectral.

\begin{lemma}\label{Lemma-prod-idem}
Let $y_1$ and $y_2$ be vertices in a graph $Y$ and $E$ be an idempotent projector for some eigenspace of $Y$.
Let $z$ be a vertex in a graph $Z$ and $\mu \in \Phi_z$ with corresponding idempotent projector $F$.

If $W$ is an idempotent projector for an eigenspace of $Y \square Z$ such that $(E\otimes F)W \neq 0$ and 
$W e_{(y_1,z)} = \alpha W e_{(y_2,z)}$ for some $\alpha\in\C$,  then $E e_{y_1} = \alpha E e_{y_2}$.
\end{lemma}
\begin{proof}
Since $(E\otimes F)W\neq 0$, we can write 
\begin{equation*}
W = E\otimes F + \sum_{j} E_{r_j} \otimes F_{s_j}
\end{equation*}
where the $E_{r_j}$'s and the $F_{s_j}$'s are idempotent projectors for $Y$ and $Z$,
respectively, different from $E$ and $F$. 

Multiplying $E\otimes F$ on the left to both sides of $W e_{(y_1,z)} = \alpha W e_{(y_2,z)}$ gives
\begin{equation*}
E e_{y_1} \otimes F e_{z} = \alpha E e_{y_2} \otimes F e_{z},
\end{equation*}
and $E e_{y_1}= \alpha E e_{y_2}$.
\end{proof}

\begin{corollary}\label{Lemma-prod-sc}
Let  $y_1$ and $y_2$ be vertices in a graph $Y$ and let $z$ be a vertex in a graph $Z$.
If $(y_1, z)$ is strongly cospectral to $(y_2,z)$ in $Y\square Z$, then $y_1$ is strongly cospectral to $y_2$ in $Y$.
\end{corollary}

\begin{lemma}\label{Lemma-gcd>=3}
If $\gcd(n+1,m+1) \geq  3$, then $(1,1)$ and $(n,1)$ are not strongly cospectral vertices in $P_n \square P_m$.
\end{lemma}
\begin{proof}
We use $E_r$ to denote the idempotent projector onto the $\left(2\cos \frac{r\pi}{n+1}\right)$-eigenspace of $A(P_n)$,
and $F_s$ to denote the idempotent projector onto the $\left(2\cos \frac{s\pi}{m+1}\right)$-eigenspace of $A(P_m)$.

Let $g=\gcd(n+1, m+1)$, $h_n = \frac{n+1}{g}$ and $h_m=\frac{m+1}{g}$.  Without loss of generality, we assume $h_n$ is odd.
Then 
\begin{equation*}
\theta = 2\cos\left(\frac{h_n\pi}{n+1}\right)+2\cos \left(\frac{2h_m\pi}{m+1}\right) = 2\cos \left(\frac{2h_n\pi}{n+1}\right)+2\cos \left(\frac{h_m\pi}{m+1}\right)
\end{equation*} 
is an eigenvalue of $P_n \square P_m$.
If $W$ is the idempotent projector of the $\theta$-eigenspace of $P_n \square P_m$, then
both $(E_{h_n}\otimes F_{2h_m} )W$ and $(E_{2h_n}\otimes F_{h_m}) W$ are non-zero.
As $E_{h_n}e_1 = E_{h_n}e_n$ and $E_{2h_n}e_1 = -E_{2h_n}e_n$, Lemma~\ref{Lemma-prod-idem} implies $W e_{(1,1)} \neq \pm W e_{(n,1)}$.
\end{proof}

\begin{corollary}\label{Corollary-gcd3-sc}
If $\gcd(n_i+1, n_j+1) \geq 3$, for some $1\leq i < j \leq n$, then the corners of $P_{n_1}\square P_{n_2}\square\cdots\square P_{n_k}$ are not mutually strongly cospectral.
\end{corollary}

\section{No pretty good state transfer among all corners}\label{nopgstsection}
We shall identify pairs of paths whose cartesian product does not have all four corners in the same PGST-equivalence class. 
If any of these pairs appears as factors in $P_{n_1}\square \dots \square P_{n-k}$, then Lemma~\ref{Lemma-1factor} impllies that the corners of $P_{n_1}\square \dots \square P_{n-k}$ are not in the same PGST-equivalence class.

The following lemma covers $P_{p-1}\square P_{p-1}$, $P_{p-1}\square P_{2p-1}$, $P_{2p-1}\square P_{2p-1}$, and $P_{2^e-1}\square P_{2^f-1}$, for $e, f\geq 2$.

\begin{lemma}\label{Lemma-gcd3-pgst}
If $\gcd(n+1, m+1) \geq 3$, then the four corners of $P_n\square P_m$ do not form a PGST-equivalence class.
\end{lemma}
\begin{proof}
    It follows immediately from Corollary~\ref{Corollary-gcd3-sc} and Theorem~\ref{PGSTcriterion}.
\end{proof}

For each pair of paths considered in this section, we give a sequence of integer $\{\ell_{rs}\}$ that satisfies the three conditions in Theorem~\ref{PGSTcriterion}(b) using the $2\times 2$ matrices
\begin{equation} \label{ABCmatrix}
A=\begin{bmatrix}1&0\\-2&1\end{bmatrix}, \quad
B=\begin{bmatrix}1&-1\\-2&2\end{bmatrix} \quad \text{and}\quad
C=\begin{bmatrix}-1&2\\1&-2\end{bmatrix}. 
\end{equation}

\subsection{$P_{p_1-1}\square P_{p_2-1}$}

For $i=1,2$, let $p_i$ be an odd prime and $\alpha^{(i)}_r = 2\cos \frac{r\pi}{p_i}$, it follows from
Equation~(\ref{Eqn-p-alt-sum}) and $\alpha^{(i)}_r= -\alpha^{(i)}_{p_i-r}$ that
\begin{equation}\label{Eqn-p-alt-sum2}
2+\sum_{r=1}^{p_i-1}(-1)^r \alpha^{(i)}_r=0.
\end{equation}

\begin{lemma}\label{Lemma-p-q-3mod4}
If $p_1 \equiv 3 \pmod 4$, then there is no pretty good state transfer from $(1,1)$ to $(p_1-1, 1)$ in $P_{p_1-1}\square P_{p_2-1}$.
\end{lemma}
\begin{proof}

Define the $(p_1-1)\times (p_2-1)$ matrix 
\begin{equation*} 
L:=\begin{bmatrix} A &  B&\cdots &B\\C& 0&\cdots &0\\\vdots &  \vdots & \ddots & \vdots\\C& 0&\cdots &0\end{bmatrix}
\end{equation*}
containing $\left(\frac{p_2-1}{2}-1\right)$ copies of $B$ and $\left(\frac{p_1-1}{2}-1\right)$ copies of $C$. 
Note that the entries of the $r$-th row of $L$ sum to $(-1)^{r+1}$ and the entries of
the $s$-th column of $L$ sum to $(-1)^s$.   
 
Let $\ell_{rs}=L_{r,s}$, for $1\leq r \leq p_1-1$ and $1\leq s\leq p_2-1$.
For Condition~(b)(i) of Theorem~\ref{PGSTcriterion}, we have
 \begin{eqnarray*}
\sum_{r=1}^{p_1-1}  \sum_{s=1}^{p_2-1}\ell_{rs}(\alpha^{(1)}_{r}+\alpha^{(2)}_{s}) 
&=& \sum_{r=1}^{p_1-1} \left( \sum_{s=1}^{p_2-1}\ell_{rs}\right) \alpha^{(1)}_{r} + \sum_{s=1}^{p_2-1} \left( \sum_{r=1}^{p_1-1} \ell_{rs} \right)\alpha^{(2)}_{s} \\
&=& \sum_{r=1}^{p_1-1} (-1)^{r+1} \alpha^{(1)}_{r} + \sum_{s=1}^{p_2-1} (-1)^{s}\alpha^{(2)}_{s},
 \end{eqnarray*}
which is $0$ as a result of Equation~(\ref{Eqn-p-alt-sum2}).
 
Since the entries of $L$ sum to zero, Condition (b)(ii) holds.

Now $\Phi^-_{(1,1), (p_1-1,1)}=\left\{\alpha^{(1)}_{r}+\alpha^{(2)}_{s} \ :\  \text{$r$ is even}\right\}$, we have
\begin{eqnarray*}
\sum_{\text{$r,s$ : $r$ even}} \ell_{rs} =\sum_{\text{$r$ even}} (-1)^{r+1}=  - \left(\frac{p_1-1}{2}\right).
\end{eqnarray*}
If $p_1 \equiv 3 \pmod 4$, then the sequence $\{\ell_{rs}\}$ satisfies Conditions (b)(i) to (b)(iii) of Theorem~\ref{PGSTcriterion}, so there is no pretty good state transfer  between $(1,1)$
and $(p_1-1,1)$ in $P_{p_1-1}\square P_{p_2-1}$.
\end{proof}

\begin{lemma}\label{Lemma-p-q-1mod4}
If $p_1 \equiv 5 \pmod 8$ and $p_2 \equiv 1 \pmod 4$, then there is no pretty good state transfer from $(1,1)$ to $(p_1-1, 1)$ in $P_{p_1-1}\square P_{p_2-1}$.
\end{lemma}
\begin{proof}
Using the $2\times 2$ matrices in Equation~(\ref{ABCmatrix}), we define the $\frac{p_1-1}{2}\times \frac{p_2-1}{2}$ matrix 
\begin{equation*} 
L:=\begin{bmatrix} A &  B&\cdots &B\\C& 0&\cdots &0\\\vdots &  \vdots & \ddots & \vdots\\C& 0&\cdots &0\end{bmatrix}
\end{equation*}
containing $\left(\frac{p_2-1}{4}-1\right)$ copies of $B$ and $\left(\frac{p_1-1}{4}-1\right)$ copies of $C$.  Let
\begin{equation*}
\ell_{rs} =
    \begin{cases}
       L_{r,s} & \text{if $1\leq r\leq \frac{p_1-1}{2}$, and $1\leq s\leq \frac{p_2-1}{2}$},\\
       0 & \text{otherwise.}
    \end{cases}
\end{equation*}

Similar to the proof of Lemma~\ref{Lemma-p-q-3mod4}, we have 
 \begin{eqnarray*}
 \sum_{r=1}^{p_1-1}  \sum_{s=1}^{p_2-1}\ell_{rs}(\alpha^{(1)}_{r}+\alpha^{(2)}_{s})= \sum_{r=1}^{\frac{p_1-1}{2}} (-1)^{r+1} \alpha^{(1)}_{r} + \sum_{s=1}^{\frac{p_2-1}{2}} (-1)^{s}\alpha^{(2)}_{s}.
 \end{eqnarray*}
It follows from Equation~(\ref{Eqn-p-alt-sum}) that this sum is equal to zero, and
Condition~(b)(i) of Theorem~\ref{PGSTcriterion} holds.
The sequence $\{\ell_{rs}\}$ also satisfies
Condition~(b)(ii)
 and
 \begin{equation*}
 \sum_{\text{$r,s$ : $r$ even}} \ell_{rs} = -\left(\frac{p_1-1}{4}\right).
 \end{equation*}
If $p_1 \equiv 5 \pmod 8$, then the sequence $\{\ell_{rs}\}$ satisfies Conditions (b)(i) to (b) (iii) of Theorem~\ref{PGSTcriterion} and there is no pretty good state transfer from $(1,1)$ to 
$(n_1,1)$.

\end{proof}

\subsection{$P_{2p_1-1}\square P_{2p_2-1}$}
For $i=1,2$, let $\beta_r^{(i)}=2\cos \frac{r\pi}{2p_i}$.  Note that $\beta_{p_i}^{(i)}=0$.
\begin{lemma}\label{Lemma-2p-2q-3mod4}
If $p_1\equiv 3 \pmod 4$, then there is no pretty good state transfer from $(1,1)$ to $(2p_1-1, 1)$ in $P_{2p_1-1}\square P_{2p_2-1}$.
\end{lemma}
\begin{proof}
We first consider the case where $p_2 \equiv 1 \pmod 4$.
Define the $\frac{p_1+1}{2} \times \frac{p_2-1}{2}$ matrix $L$ as
\begin{equation*}
\begin{matrix}
& \gr{\begin{matrix}\quad \beta_2^{(2)}\ \beta_4^{(2)} &&  \cdots &&  \cdots && \beta_{p_2-3}^{(2)}\ \beta_{p_2-1}^{(2)}\end{matrix}}\\
\\
{\renewcommand{\arraystretch}{1.25} \gr{\begin{matrix} \beta_2^{(1)}\\ \beta_4^{(1)}\\ \\ \vdots \\ \\ \vdots \\ \\ \beta_{p_1-5}^{(1)}\\ \beta_{p_1-3}^{(1)} \\  \beta_{p_1-1}^{(1)} \\ \beta_{p_1}^{(1)}\end{matrix}}} & 
\begin{bNiceMatrix}
\Block{3-4}{A} &&&& \Block{3-4}{B}&&&& \Block{3-4}{\cdots}&&&& \Block{3-4}{B}&&&\\
&&&&&&&&&&&&&&&\\
&&&&&&&&&&&&&&&\\
\Block{3-4}{C} &&&& \Block{3-4}{0}&&&& \Block{3-4}{\cdots}&&&& \Block{3-4}{0}&&&\\
&&&&&&&&&&&&&&&\\
&&&&&&&&&&&&&&&\\
\Block{3-4}{\vdots} &&&& \Block{3-4}{\vdots}&&&& \Block{3-4}{\ddots}&&&& \Block{3-4}{\vdots}&&&\\
&&&&&&&&&&&&&&&\\
&&&&&&&&&&&&&&&\\
\Block{3-4}{C} &&&& \Block{3-4}{0}&&&& \Block{3-4}{\cdots}&&&& \Block{3-4}{0}&&&\\
&&&&&&&&&&&&&&&\\
&&&&&&&&&&&&&&&\\
\Block{3-4}{C} &&&& \Block{3-4}{0}&&&& \Block{3-4}{\cdots}&&&& \Block{3-4}{0}&&&\\
&&&&&&&&&&&&&&&\\
&&&&&&&&&&&&&&&\\
\end{bNiceMatrix}
\end{matrix}
\end{equation*}
containing $\left(\frac{p_2-1}{4}-1\right)$ copies of $B$
and $\left(\frac{p_1+1}{4}-1\right)$ copies of $C$.
Let
\begin{equation*}
\ell_{rs} =
    \begin{cases}
       L_{\frac{r}{2},\frac{s}{2}} & \text{if $r$ is even with $2\leq r\leq p_1-1$, and $s$ is even with $2\leq s\leq p_2-1$},\\
       L_{\frac{p_1+1}{2},\frac{s}{2}} &\text{if $r=p_1$, and $s$ is even with $2 \leq s \leq p_2-1$},\\
       0 & \text{otherwise.}
    \end{cases}
\end{equation*}
(We list on the right of $L$ the eigenvalues of $P_{2p_1-1}$ associated with each row of $L$, and above $L$ the eigenvalues of $P_{2p_2-1}$ associated with each column of $L$.)

For Conditon~(b)(i) in Theorem~\ref{PGSTcriterion}, we have
\begin{eqnarray*}
&&\sum_{r=1}^{2p_1-1}\sum_{s=1}^{2p_2-1} \ell_{rs}\left(\beta_r^{(1)}+\beta_s^{(2)}\right)\\
&=& \sum_{j=1}^{\frac{p_1-1}{2}} \left(\sum_{k=1}^{\frac{p_2-1}{2}}L_{j,k}\right) \beta_{2j}^{(1)} + 
\left(\sum_{k=1}^{\frac{p_2-1}{2}}L_{\frac{p_1+1}{2},k} \right) \beta_{p_1}^{(1)} +
\sum_{k=1}^{\frac{p_2-1}{2}} \left(\sum_{j=1}^{\frac{p_1+1}{2}}L_{j,k}\right) \beta_{2k}^{(2)}\\
&=& \sum_{j=1}^{\frac{p_1-1}{2}} (-1)^{j+1}\beta_{2j}^{(1)} + 0 + \sum_{k=1}^{\frac{p_2-1}{2}} (-1)^{k}\beta_{2k}^{(2)}.
\end{eqnarray*}
which is equal to $0$ by Equation~(\ref{Eqn-2p-alt-sum}).

It is straightforward to check that $\sum_{r,s} \ell_{rs} =0$.   Since 
\begin{equation*}
    \Phi^-_{(1,1),(2p_1-1,1)} = \{\beta_{2j}^{(1)}+\beta_s^{(2)} :  1\leq j \leq p_1-1, 1\leq s \leq 2p_2-1\},
\end{equation*} 
we have
\begin{eqnarray*}
\sum_{j=1}^{p_1-1}\sum_{s=1}^{2p_2-1}l_{2j,s}
=\sum_{j=1}^{\frac{p_1-1}{2}}\sum_{k=1}^{\frac{p_2-1}{2}}L_{j,k}
= \sum_{j=1}^{\frac{p_1-1}{2}} (-1)^{j+1} \equiv 1 \pmod 2.
\end{eqnarray*}
It follows from Theorem~\ref{PGSTcriterion} that there is no pretty good state transfer from $(1,1)$ to $(2p_1-1,1)$.

When $p_2\equiv 3\pmod 4$,
we define the $\frac{p_1+1}{2} \times \frac{p_2+1}{2}$ matrix $L$ as
\begin{equation*}
\begin{matrix}
& \gr{\begin{matrix}\quad \beta_2^{(2)}\ \beta_4^{(2)} &&&  \cdots &&  \cdots &&& \beta_{p_2-5}^{(2)}\ \beta_{p_2-3}^{(2)} & \beta_{p_2-1}^{(2)}\ \beta_{p_2}^{(2)}\end{matrix}}\\
\\
{\renewcommand{\arraystretch}{1.25} \gr{\begin{matrix} \beta_2^{(1)}\\ \beta_4^{(1)}\\ \\ \vdots \\ \\ \vdots \\ \\ \beta_{p_1-5}^{(1)}\\ \beta_{p_1-3}^{(1)} \\  \beta_{p_1-1}^{(1)} \\ \beta_{p_1}^{(1)}\end{matrix}}} & 
\begin{bNiceMatrix}
\Block{3-4}{A} &&&&& \Block{3-4}{B}&&&&& \Block{3-4}{\cdots}&&&&& \Block{3-4}{B}&&&&& \Block{3-4}{B}&&&\\
&&&&&&&&&&&&&&&\\
&&&&&&&&&&&&&&&\\
\Block{3-4}{C} &&&&& \Block{3-4}{0}&&&&& \Block{3-4}{\cdots}&&&&& \Block{3-4}{0}&&&&& \Block{3-4}{0}&&&\\
&&&&&&&&&&&&&&&\\
&&&&&&&&&&&&&&&\\
\Block{3-4}{\vdots} &&&&& \Block{3-4}{\vdots}&&&&& \Block{3-4}{\ddots} &&&&& \Block{3-4}{\vdots}&&&&& \Block{3-4}{\vdots}&&&\\
&&&&&&&&&&&&&&&\\
&&&&&&&&&&&&&&&\\
\Block{3-4}{C} &&&&& \Block{3-4}{0}&&&&& \Block{3-4}{\cdots} &&&&& \Block{3-4}{0}&&&&& \Block{3-4}{0}&&&\\
&&&&&&&&&&&&&&&\\
&&&&&&&&&&&&&&&\\
\Block{3-4}{C} &&&&& \Block{3-4}{0}&&&&& \Block{3-4}{\cdots}&&&&& \Block{3-4}{0}&&&&& \Block{3-4}{0}&&&\\
&&&&&&&&&&&&&&&\\
&&&&&&&&&&&&&&&\\
\end{bNiceMatrix}
\end{matrix}
\end{equation*}
containing $\left(\frac{p_2+1}{4}-1\right)$ copies of $B$
and $\left(\frac{p_1+1}{4}-1\right)$ copies of $C$.
Let
\begin{equation*}
\ell_{rs} =
    \begin{cases}
       L_{\frac{r}{2},\frac{s}{2}} & \text{if $r$ is even with $2\leq r\leq p_1-1$, and $s$ is even with $2\leq s\leq p_2-1$},\\
       L_{\frac{p_1+1}{2},\frac{s}{2}} &\text{if $r=p_1$, and $s$ is even with $2\leq s\leq p_2-1$},\\
       L_{\frac{r}{2},\frac{p_2+1}{2}} &\text{if $s=p_2$, and $r$ is even with $2\leq r\leq p_1-1$},\\
       0 & \text{otherwise.}
    \end{cases}
\end{equation*}
For Conditon~(b)(i) in Theorem~\ref{PGSTcriterion}, we have
\begin{eqnarray*}
&&\sum_{r=1}^{2p_1-1}\sum_{s=1}^{2p_2-1} \ell_{rs}\left(\beta_r^{(1)}+\beta_s^{(2)}\right)\\
&=& \sum_{j=1}^{\frac{p_1-1}{2}} \left(\sum_{k=1}^{\frac{p_2+1}{2}}L_{j,k}\right) \beta_{2j}^{(1)} + 
\left(\sum_{k=1}^{\frac{p_2+1}{2}}L_{\frac{p_1+1}{2},k} \right) \beta_{p_1}^{(1)} +
\sum_{k=1}^{\frac{p_2-1}{2}} \left(\sum_{j=1}^{\frac{p_1+1}{2}}L_{j,k}  \right) \beta_{2k}^{(2)} + 
\left(\sum_{j=1}^{\frac{p_1+1}{2}}L_{j,\frac{p_2+1}{2}} \right) \beta_{p_2}^{(2)}\\
&=& \sum_{j=1}^{\frac{p_1-1}{2}} (-1)^{j+1}\beta_{2j}^{(1)} + 0 + \sum_{k=1}^{\frac{p_2-1}{2}} (-1)^{k}\beta_{2k}^{(2)}+0\\
&=&0.
\end{eqnarray*}
We also have $\sum_{r,s}\ell_{rs} = \sum_{r,s} L_{r,s}=0$, and
\begin{eqnarray*}
\sum_{j=1}^{p_1-1}\sum_{s=1}^{2p_2-1}l_{2j,s}
=\sum_{j=1}^{\frac{p_1-1}{2}}\sum_{k=1}^{\frac{p_2+1}{2}}L_{j,k}
= \sum_{j=1}^{\frac{p_1-1}{2}} (-1)^{j+1} \equiv 1 \pmod 2.
\end{eqnarray*}
It follows from Theorem~\ref{PGSTcriterion} that  pretty good state transfer does not occur between $(1,1)$ and $(2p_1-1,1)$.

\end{proof}

\subsection{$P_{2p_1-1}\square P_{p_2-1}$}

For $r=1, \dots, 2p_1-1$, let $\beta_r = 2\cos \frac{r\pi}{2p_1}$ be the eigenvalues of $P_{2p_1-1}$.
For $s=1, \dots, p_2-1$, let $\alpha_s = 2\cos \frac{s\pi}{p_2}$ be the eigenvalues of $P_{p_2-1}$.

\begin{lemma}\label{Lemma-2p-q-3mod4-1mod4}
If $p_1 \equiv 3 \pmod 4$ and $p_2 \equiv 1 \pmod 4$, then there is no pretty good state transfer from $(1,1)$ to $(2p_1-1, 1)$ in $P_{2p_1-1}\square P_{p_2-1}$.
\end{lemma}
\begin{proof}
Define the $\frac{p_1+1}{2}\times \frac{p_2-1}{2}$ matrix $L$ as
\begin{equation*}
\begin{matrix}
& \gr{\begin{matrix}\quad \alpha_1\ \alpha_2 &&  \cdots &&  \cdots && \alpha_{\frac{p_2-3}{2}}\ \alpha_{\frac{p_2-1}{2}}\end{matrix}}\\
\\
\gr{\begin{matrix}\beta_2\\ \beta_4\\ \\\vdots \\ \\ \vdots\\ \\ \beta_{(p_1-1)/2} \\ \beta_{p_1}\end{matrix}} & 
\begin{bNiceMatrix}
\Block{3-4}{A} &&&& \Block{3-4}{B}&&&& \Block{3-4}{\cdots}&&&& \Block{3-4}{B}&&&\\
&&&&&&&&&&&&&&&\\
&&&&&&&&&&&&&&&\\
\Block{3-4}{C} &&&& \Block{3-4}{0}&&&& \Block{3-4}{\cdots}&&&& \Block{3-4}{0}&&&\\
&&&&&&&&&&&&&&&\\
&&&&&&&&&&&&&&&\\
\Block{3-4}{\vdots} &&&& \Block{3-4}{\vdots}&&&& \Block{3-4}{\ddots}&&&& \Block{3-4}{\vdots}&&&\\
&&&&&&&&&&&&&&&\\
&&&&&&&&&&&&&&&\\
\Block{3-4}{C} &&&& \Block{3-4}{0}&&&& \Block{3-4}{\cdots}&&&& \Block{3-4}{0}&&&\\
&&&&&&&&&&&&&&&\\
&&&&&&&&&&&&&&&\\
\end{bNiceMatrix}
\end{matrix}
\end{equation*}
containing $\left(\frac{p_2-1}{4}-1\right)$ copies of $B$ and $\left(\frac{p_1+1}{4}-1\right)$ copies of $C$.  Let
\begin{equation*}
    \ell_{rs}=
       \begin{cases}
       L_{\frac{r}{2},s} & \text{if $r$ is even with $2\leq r\leq p_1-1$, and $1\leq s\leq \frac{p_2-1}{2}$},\\
       L_{\frac{p_1+1}{2},s} &\text{if $r=p_1$ and  $1\leq s\leq \frac{p_2-1}{2}$},\\
       0 & \text{otherwise.}
    \end{cases}
\end{equation*}

For Conditon~(b)(i) in Theorem~\ref{PGSTcriterion}, we have
\begin{eqnarray*}
&&\sum_{r=1}^{2p_1-1}\sum_{s=1}^{p_2-1} \ell_{rs}\left(\beta_r +\alpha_s\right)\\
&=& \sum_{j=1}^{\frac{p_1-1}{2}}\left(\sum_{s=1}^{\frac{p_2-1}{2}} L_{j,s}\right)\beta_{2j} + \left(\sum_{s=1}^{\frac{p_2-1}{2}} L_{\frac{p_1+1}{2},s}\right)\beta_{p_1} + \sum_{s=1}^{\frac{p_2-1}{2}}\left(\sum_{j=1}^{\frac{p_1+1}{2}} L_{j,s}\right)\alpha_{s}\\
&=& \sum_{j=1}^{\frac{p_1-1}{2}}(-1)^{j+1}\beta_{2j} + 0 + 
\sum_{s=1}^{\frac{p_2-1}{2}} (-1)^s\alpha_{s}\\
&=&0.
\end{eqnarray*}
It is straightforward that Condition~(b)(ii) holds.
For Condition~(b)(iii), 
\begin{equation*}
    \Phi^-_{(1,1),(2p_1-1,1)}=\{\beta_r+\alpha_s : \text{$r$ is even}\},
\end{equation*} 
and
\begin{eqnarray*}
    \sum_{\text{$r$ is even}}\sum_{s=1}^{p_2-1} \ell_{rs} = \sum_{j=1}^{\frac{p_1-1}{2}}\sum_{s=1}^{p_2-1}L_{j,s} = \sum_{j=1}^{\frac{p_1-1}{2}}(-1)^{j+1}=1 \pmod 2.
\end{eqnarray*}
By Theorem~\ref{PGSTcriterion}, there is no pretty good state transfer between $(1,1)$ and $(2p_1-1,1)$ in
$P_{2p_1-1}\square P_{p_2-1}$.
\end{proof}

\begin{lemma}\label{Lemma-2p-q-all-3mod4}
If  $p_2 \equiv 3 \pmod 4$, then there is no pretty good state transfer from $(1,1)$ to $(1, p_2-1)$ in $P_{2p_1-1}\square P_{p_2-1}$.
\end{lemma}
\begin{proof}
Define the $\left(p_1-1\right)\times \left(p_2-1\right)$ matrix
\begin{equation*}
L:=\begin{bmatrix} A & B& \cdots & B\\C&0&\cdots&0\\ \vdots& \vdots & \ddots & \vdots\\ C&0&\cdots&0\end{bmatrix}
\end{equation*}
containing $\left(\frac{p_2-1}{2}-1\right)$ copies of $B$ and $\left(\frac{p_1-1}{2}-1\right)$ copies of $C$.
Let
\begin{equation*}
    \ell_{rs}=
       \begin{cases}
       L_{\frac{r}{2},s} & \text{if $r$ is even with $2\leq r\leq 2p_1-2$, and  $1\leq s\leq p_2-1$},\\
       0 & \text{otherwise.}
    \end{cases}
\end{equation*}
For Condition~(b)(i) of Theorem~\ref{PGSTcriterion},
\begin{eqnarray*}
\sum_{r=1}^{2p_1-1}\sum_{s=1}^{p_2-1} \ell_{rs}\left(\beta_r +\alpha_s\right)
&=& \sum_{j=1}^{p_1-1}\left(\sum_{s=1}^{p_2-1}L_{j,s}\right)\beta_{2j} + \sum_{s=1}^{p_2-1}\left(\sum_{j=1}^{p_1-1} L_{j,s}\right)\alpha_{s}\\
&=& \sum_{j=1}^{p_1-1}(-1)^{j+1}\beta_{2j} + 
\sum_{s=1}^{p_2-1} (-1)^s\alpha_{s}\\
&=& 2\sum_{j=1}^{\frac{p_1-1}{2}}(-1)^{j+1}\beta_{2j} + 
\sum_{s=1}^{p_2-1} (-1)^s\alpha_{s}\\
&=&0.
\end{eqnarray*}
As the entries of $L$ sum to $0$, Condition~(b)(ii) holds.

For Condition~(b)(iii), 
\begin{equation*}
    \Phi^-_{(1,1), (1,p_2-1)} = \{\beta_r+\alpha_s : \text{$s$ is even}\}
\end{equation*}
and 
\begin{equation*}
    \sum_{\text{$s$ is even}} \sum_{r=1}^{2p_1-1}\ell_{rs} =  \sum_{\text{$s$ is even}} \sum_{j=1}^{p_1-1}L_{j,s} = \sum_{\text{$s$ is even}} (-1)^s = \frac{p_2-1}{2} \equiv 1 \pmod{2}.
\end{equation*}
By Theorem~\ref{PGSTcriterion}, there is no pretty good state transfer between $(1,1)$ and $(1, p_2-1)$ in
$P_{2p_1-1}\square P_{p_2-1}$.
\end{proof}

\begin{lemma}\label{Lemma-2p-q-1mod4-5mod8}
If $p_1 \equiv 1 \pmod 4$ and $p_2\equiv 5 \pmod 8$, then there is no pretty good state transfer from $(1,1)$ to $(1,p_2-1)$ in $P_{2p_1-1}\square P_{p_2-1}$.
\end{lemma}
\begin{proof}
Define the $\frac{p_1-1}{2}\times \frac{p_2-1}{2}$ matrix
\begin{equation*}
L:=\begin{bmatrix} A & B& \cdots & B\\C&0&\cdots&0\\ \vdots& \vdots & \ddots & \vdots\\ C&0&\cdots&0\end{bmatrix}
\end{equation*}
containing $\left(\frac{p_2-1}{4}-1\right)$ copies of $B$
and $\left(\frac{p_1-1}{4}-1\right)$ copies of $C$.
Let
\begin{equation*}
\ell_{rs} =
    \begin{cases}
       L_{\frac{r}{2},s} & \text{if $r$ is even with $2\leq r\leq p_1-1$, and $1\leq s\leq \frac{p_2-1}{2}$},\\
       0 & \text{otherwise.}
    \end{cases}
\end{equation*}
For Condition~(b)(i) of Theorem~\ref{PGSTcriterion},
\begin{eqnarray*}
\sum_{r=1}^{2p_1-1}\sum_{s=1}^{p_2-1} \ell_{rs}\left(\beta_r +\alpha_s\right)
&=& \sum_{j=1}^{\frac{p_1-1}{2}}\left(\sum_{s=1}^{\frac{p_2-1}{2}} L_{j,s}\right)\beta_{2j} + \sum_{s=1}^{\frac{p_2-1}{2}}\left(\sum_{j=1}^{\frac{p_1-1}{2}} L_{j,s}\right)\alpha_{s}\\
&=& \sum_{j=1}^{\frac{p_1-1}{2}}(-1)^{j+1}\beta_{2j} + 
\sum_{s=1}^{\frac{p_2-1}{2}} (-1)^s\alpha_{s}\\
&=&0.
\end{eqnarray*}
As the entries of $L$ sum to $0$, Condition~(b)(ii) holds.

For Condition~(b)(iii), 
\begin{equation*}
    \Phi^-_{(1,1), (1,p_2-1)} = \{\beta_r+\alpha_s : s \text{ is even}\}
\end{equation*}
and 
\begin{equation*}
    \sum_{\text{$s$ is even}} \sum_{r=1}^{2p_1-1}\ell_{rs} =  \sum_{\text{$s$ is even}} \sum_{j=1}^{\frac{p_1-1}{2}}L_{j,s} = \sum_{\text{$s$ is even}} (-1)^s = \frac{p_2-1}{4} \equiv 1 \pmod{2}.
\end{equation*}
By Theorem~\ref{PGSTcriterion}, there is no pretty good state transfer between $(1,1)$ and $(1, p_2-1)$ in
$P_{2p_1-1}\square P_{p_2-1}$.
\end{proof}

\section{Pretty good state transfer among corner vertices}
\label{pgstsection}
In this section we shall classify the path products 
where pretty good state transfer occurs among all corners.
In order to prove that there is pretty good state transfer between an arbitrary pair of corners in a given path product, it is enough to fix a corner and show that it has pretty good state transfer to
each adjacent corner. This is because pretty good state transfer defines an equivalence relation on the set of vertices. Further, as the automorphism group acts transitively on the set of corners,
we may assume that the fixed corner is $(1,1,\ldots,1)$, so it is enough to prove  pretty good state transfer between $(1,1,\ldots,1)$ and $(1,\ldots,n_i,\ldots,1)$ for all $i$. For convenience of notation, we may rearrange the cartesian factors so that the factor of interest is the first one.

\subsection{$X=P_{p-1}\square Z$}
\begin{lemma}\label{Lemma-pgstPZ} Let $p\equiv 1\pmod8$. Let $Z$ be a finite graph and denote by
$F_Z$ the field generated by its eigenvalues. Assume that $F_p\cap F_Z=\Q$. Then if, for some vertex $z$ of $Z$ the vertices
$(1,z)$ and $(p-1,z)$ in $P_{p-1}\square Z$ are strongly cospectral,
there is pretty good state transfer between $(1,z)$ and $(p-1,z)$.  
\end{lemma}
\begin{proof}
We recall that the eigenvalues of $P_{p-1}$ are 
$\alpha_r=2\cos{\frac {r\pi}{p}}$, $r=1$,\dots, $p-1$.
We shall apply Theorem~\ref{PGSTcriterion} to the vertices
$(1,z)$ and $(p-1,z)$. By assumption, Condition (a) of Theorem~\ref{PGSTcriterion} holds. 
Let $\mu_1$, \dots, $\mu_m$ be the eigenvalues of $Z$.
Then the eigenvalues of $P_{p-1}\square Z$ are the values
$\alpha_r+\mu_j$  for $1\le r\leq p-1$ and $1\le j\le m$.
We shall assume that there is a sequence of $(p-1)m$ integers $\ell_{rj}$ that satisfy Conditions (b)(i) and (b)(ii) of Theorem~\ref{PGSTcriterion} and prove that Condition (b)(iii) is impossible.

Condition (b)(i) takes the form 
\begin{equation}\label{eigen_eqn}
  \sum_{r,j} \ell_{rj}(\alpha_{r}+\mu_j)=0
\end{equation}
and Condition (b)(ii) is
\begin{equation}\label{coeff_sum}
    \sum_{r,j} \ell_{rj}=0.
\end{equation}
The set $\Phi^-_{(1,z), (p-1,z)}$ consists of
those eigenvalues $\alpha_{r}+\mu_j$ for which $r$ is even.
Therefore Condition (b)(iii) is
\begin{equation*}\label{minus_sum}
    \sum_{\text{$r,j$: $r$ even}} \ell_{rj}\equiv1\pmod2.
\end{equation*}
For $r=1$,\dots, $p-1$, let $a_r=\sum_j\ell_{rj}$.
Then we can we rewrite \eqref{eigen_eqn} as
\begin{equation}\label{eigen_eqn2}
\sum_{r=1}^{p-1}a_r\alpha_r= - \sum_{r,j} \ell_{rj}\mu_j.
\end{equation}
The left hand side of \eqref{eigen_eqn2} lies in $F_p$, while the right hand side lies in $F_Z$. So by our hypothesis
on the intersection of these fields, the common value of \eqref{eigen_eqn2} must be rational
and, in fact, an integer, since it is clearly an algebraic integer. We denote this integer by $s$. 
Then we have 
\begin{equation}\label{peqn}
  \sum_{r=1}^{p-1}a_r\alpha_r= \sum_{r=1}^{\frac{p-1}{2}} (a_r - a_{p-r}) \alpha_r = s.
\end{equation}
Using Equation~(\ref{Eqn-p-alt-sum}), we can replace $1$ in $\{1, \alpha_1, \dots, \alpha_{\frac{p-3}{2}}\}$ with $\alpha_{\frac{p-1}{2}}$ to form a basis $\{\alpha_1,\dots, \alpha_{\frac{p-1}{2}}\}$ of $F_p$.
Then Equations~(\ref{Eqn-p-alt-sum}) and (\ref{peqn}) give 
\begin{equation*}
    \sum_{r=1}^{\frac{p-1}{2}} \left(a_r-a_{p-r} + (-1)^rs\right) \alpha_r=0,
\end{equation*}
and $a_r-a_{p-r} = (-1)^{r+1}s$, for $r=1,\dots,\frac{p-1}{2}$.

Thus, for $i=1,\dots,\frac{(p-1)}{2},$
\begin{equation*}
a_{2i}-a_{p-2i}=-s.
\end{equation*}
Summing over $i$ yields
\begin{equation*}
  \sum_{\text{$r,j$: $r$  even}} \ell_{rj} -\sum_{\text{$r,j$: $r$  odd}} \ell_{rj}=-\frac{(p-1)}2s,
\end{equation*}
and if we add this equation to \eqref{coeff_sum} we obtain
\begin{equation*}
  2\sum_{\text{$r,j$: $r$ even}} \ell_{rj}=-\frac{(p-1)}2s.
\end{equation*}

Thus, since $p_1\equiv1\pmod8$, we have $\sum_{\text{$r,j$: $r$  even}} \ell_{rj}\equiv0\pmod2$. Therefore, we have shown that whenever the first two conditions (b)(i)
and (b)(ii) of Theorem~\ref{PGSTcriterion} are true the third condition (b)(iii) is false.
\end{proof}

\subsection{$X=P_{2p-1}\square Z$}
\begin{lemma}\label{Lemma-pgst2PZ} Let $p\equiv 1\pmod4$. Let $Z$ be a finite graph and denote by
$F_Z$ the field generated by its eigenvalues. Assume that $F_{2p}\cap F_Z=\Q$. Then if, for some vertex $z$ of $Z$ the vertices
$(1,z)$ and $(2p-1,z)$ in $P_{2p-1}\square Z$ are strongly cospectral,
there is pretty good state transfer between $(1,z)$ and $(2p-1,z)$.  
\end{lemma}
  \begin{proof}
We recall that the eigenvalues of $P_{2p-1}$ are 
$\beta_r=2\cos{\frac {r\pi}{2p}}$, $r=1$,\dots, $2p-1$.
Let $\mu_1$, \dots, $\mu_m$ be the eigenvalues of $Z$.
Similar to the proof of Lemma~\ref{Lemma-pgstPZ}, we
shall assume that there is a sequence of $(2p-1)m$ integers $\ell_{rj}$ that satisfy Conditions (b)(i) and (b)(ii) of Theorem~\ref{PGSTcriterion} and prove that Condition (b)(iii) is impossible.

Condition (b)(i) takes the form 
\begin{equation}\label{eigen_eqn_2p}
  \sum_{r,j} \ell_{rj}(\beta_{r}+\mu_j)=0
\end{equation}
and Condition (b)(ii) is
\begin{equation*}\label{coeff_sum_2p}
    \sum_{r,j} \ell_{rj}=0.
\end{equation*}
The set $\Phi^-_{(1,z), (2p-1,z)}$ consists of
those eigenvalues $\beta_{r}+\mu_j$ for which $r$ is even.
Therefore Condition (b)(iii) is
\begin{equation*}\label{minus_sum_2p}
    \sum_{\text{$r,j$: $r$ even}} \ell_{rj}\equiv1\pmod2.
\end{equation*}
For $r=1$,\dots, $2p-1$, let $a_r=\sum_j\ell_{rj}$.
Then we can we rewrite \eqref{eigen_eqn_2p} as
\begin{equation}\label{eigen_eqn2_2p}
\sum_{r=1}^{2p-1}a_r\beta_r= - \sum_{r,j} \ell_{rj}\mu_j.
\end{equation}
The left hand side of \eqref{eigen_eqn2_2p} lies in $F_{2p}$, while the right hand side lies in $F_Z$. So by our hypothesis
on the intersection of these fields, the common value of \eqref{eigen_eqn2_2p} must be an integer, denoted by $s$. 
As $\beta_p=0$, we have 
\begin{equation}\label{peqn_2p}
  \sum_{r=1}^{2p-1}a_r\beta_r= \sum_{r=1}^{p-1} (a_r - a_{2p-r}) \beta_r = s.
\end{equation}
Using Equation~(\ref{Eqn-2p-alt-sum}), we can replace $1$ in $\{1, \beta_1, \dots, \beta_{p-2}\}$ with $\beta_{p-1}$ to form a basis $\{\beta_1,\dots, \beta_{p-1}\}$ of $F_{2p}$.
It follows from Equations~(\ref{Eqn-2p-alt-sum}) and (\ref{peqn_2p}) that 
\begin{equation*}
    \sum_{j=1}^{\frac{p-1}{2}} \left(a_{2j}-a_{2p-2j} + (-1)^js\right) \beta_{2j} + \sum_{j=1}^{\frac{p-1}{2}} (a_{2j+1} - a_{2p-2j-1}) \beta_{2j+1} =0,
\end{equation*}
and the coefficients of the $\beta_r$'s are zero.
For $i=1,\dots,\frac{(p-1)}{2}$, we have
\begin{equation*}
a_{2i}-a_{2p-2i}=(-1)^{i+1} s.
\end{equation*}
Since $\frac{p-1}{2}$ is even, summing over $i$ yields
\begin{equation*}
  \sum_{i=1}^{\frac{p-1}{2}}a_{2i} - \sum_{i=\frac{p+1}{2}}^{p-1} a_{2i}=0,
\end{equation*}
and 
\begin{equation*}
  \sum_{\text{$r,j$: $r$ even}} \ell_{rj}=\sum_{i=1}^{\frac{p-1}{2}}a_{2i} + \sum_{j=\frac{p+1}{2}}^{p-1} a_{2j} = 2 \left(\sum_{j=\frac{p+1}{2}}^{p-1} a_{2j}\right) \equiv 0 \pmod 2.
\end{equation*}
Therefore, we have shown that whenever the first two conditions (b)(i)
and (b)(ii) of Theorem~\ref{PGSTcriterion} are true the third condition (b)(iii) is false.
\end{proof}
\subsection{$X=P_{2^e-1}\square Z$}

Let $\gamma_r := 2\cos \frac{r\pi}{2^e}$ ($r=1\dots, 2^e-1$) be the eigenvalues of $P_{2^e-1}$.
  \begin{lemma}\label{Lemma-pgst2Z} 
Let $Z$ be a finite graph and denote by
$F_Z$ the field generated by its eigenvalues. Assume that $F_{2^e}\cap F_Z=\Q$. Then if, for some vertex $z$ of $Z$ the vertices
$(1,z)$ and $(2^e-1,z)$ in $P_{2^e-1}\square Z$ are strongly cospectral,
there is pretty good state transfer between $(1,z)$ and $(2^e-1,z)$.  
\end{lemma}
\begin{proof}
Let $\mu_1$,\dots,$\mu_m$ be the eigenvalues of $Z$.
We apply the criteria of Theorem~\ref{PGSTcriterion}. The strong
cospectrality condition (a) of that theorem holds by assumption
The eigenvalues in $\Phi^-_{(1,z),(2^e-1,z)}$ are $\gamma_r+\mu_j$ with $r$ even.
Suppose Condition~(b)(i) of Theorem~\ref{PGSTcriterion} holds. Then, as in previous arguments, we isolate the terms coming from eigenvalues of $P_{2^e-1}$ and use the hypothesis $F_{2^e}\cap F_Z=\Q$,
and obtain an equation
\begin{equation*}
\sum_{r=1}^{2^e-1} a_r\gamma_r=s,
\end{equation*}
where $a_i$, $s\in \Z$.
Note that $\gamma_{2^{e-1}}=0$. So we have
\begin{equation*}
\sum_{r=1}^{2^{e-1}-1} a_r \gamma_r +\sum_{r=2^{e-1}+1}^{2^e-1} a_r \gamma_r=s. 
\end{equation*}
Since $\gamma_{2^e-r}=-\gamma_r$, the above equation becomes
\begin{equation*}
\sum_{r=1}^{2^{e-1}-1} (a_r-a_{2^e-r})\gamma_r=s. 
\end{equation*}
As $\gamma_1$ has degree $2^{e-1}$, the set $\{1,\gamma_1, \dots, \gamma_{2^{e-1}-1}\}$ is linearly independent.
Hence $a_r= a_{2^e-r}$, for $r=1,\dots 2^{e-1}$, and $s=0$.
It follows that $\sum_{\text{$r$ odd}} a_r$  is even. If we assume
     Condition~(b)(ii) in Theorem~\ref{PGSTcriterion}, then $\sum_{r=1}^{2^e-1}a_r=0.$ 
     Therefore $\sum_{\text{$r$ even}} a_r$  is even.
     Then by Theorem~\ref{PGSTcriterion} we have pretty good state transfer between $(1,z)$ and
     $(2^e-1,z)$.
\ignore{
    Let $\om=e^{\frac{2\pi i}{2^{e+1}}}$, a primitive $2^{e+1}$-th root of unity. The eigenvalues of $P_{2^e-1}$ are $\om^r+\om^{-r}$, for $r=1$,\dots, $2^e-1$.
    Let $\mu_1$,\dots,$\mu_m$ be the eigenvalues of $Z$.
    We apply the criteria of Theorem~\ref{PGSTcriterion}. The strong
    cospectrality condition (a) of that theorem holds by assumption
    The eigenvalues in
    $\Phi^-_{(1,z),(2^e-1,z)}$ are the $(\om^r+\om^{-r})+\mu_j$ with $r$ even.
    Suppose equation (b)(i) of Theorem~\ref{PGSTcriterion} holds. Then, as in previous arguments, we isolate the terms coming from eigenvalues of $P_{2^e-1}$ and use the hypothesis $F_{2^e}\cap F_Z=\Q$,
    resulting in an equation
    \begin{equation}\sum_{r=1}^{2^e-1} a_r(\om^r+\om^{-r})=s,
    \end{equation}
    where $a_i$, $s\in \Z$.
    Note that for $r=2^{e-1}$, we have $\om^r+\om^{-r}=i+(-i)=0$. So we have

    \begin{equation}\sum_{r=1}^{2^{e-1}-1} a_r(\om^r+\om^{-r})+\sum_{r=2^{e-1}+1}^{2^e-1} a_r(\om^r+\om^{-r})=s. 
    \end{equation}
    Set $j=2^e-r$ in the second sum and rearrange, using
      $\om^{2^e}=-1$ to get
     \begin{equation}\sum_{r=1}^{2^{e-1}-1} (a_r-a_{2^e-r})(\om^r+\om^{-r})=s. 
     \end{equation}
     Multiply by $\om^{2^{e-1}}$ to get
     \begin{equation}\sum_{r=1}^{2^{e-1}-1} (a_r-a_{2^e-r})(\om^{2^{e-1}+r}+
       \om^{2^{e-1}-r})-s\om^{2^{e-1}}=0. 
     \end{equation}
     This is a polynomial equation in $\om$ of degree no greater than $2^e-1$. Since the minimum polynomial of $\om$ over $\Q$ is $x^{2^e}+1$, all coefficients  $a_r-a_{2^e-r}$ and $s$ must be zero. It follows that $\sum_{\text{$r$ odd}} a_r$  is even. If we assume
     condition Theorem~\ref{PGSTcriterion}(b)(ii), then $\sum_{r=1}^{2^e-1}a_r=0.$ 
     Therefore $\sum_{\text{$r$ even}} a_r$  is even.
     Then by Theorem~\ref{PGSTcriterion} we have pretty good state transfer between $(1,z)$ and
     $(2^e-1,z)$.
 }    
\end{proof}
\begin{lemma}\label{Lemma-pgst2e-p} 
For prime $p\geq 3$ and $e\geq 2$, there is pretty good state transfer between $(1,1)$ and $(1,p-1)$ in $P_{2^e-1}\square P_{p-1}$.
\end{lemma}
\begin{proof}
   We shall apply Theorem~\ref{PGSTcriterion}.
    The condition on strong cospectrality is satisfied, by Corollary~\ref{Lemma-scprods}.
    Recall the eigenvalues of $P_{p-1}$ are  $\alpha_r=2\cos(\frac{r\pi}{2p})$, $r=1$,\dots,$p-1$.
    The set $\Phi^-_{(1,1), (1,p-1)}$ consists of the eigenvalues
    $\gamma_i+\alpha_r$ where
    $r$ is even.

     Suppose Condition~(b)(i)
    of Theorem~\ref{PGSTcriterion} holds.
    Thus for some integers $\ell_{ir}$, we have
    \begin{equation*}
        \sum_{i,r}\ell_{i,r}(\gamma_i+\alpha_r)=0
    \end{equation*}
    Then, as in Lemma~\ref{Lemma-pgst2Z} we
    move the eigenvalues of $P_{2^e-1}$
    to the left side and those of
    $P_{p-1}$ to the right side, 
    resulting in the equation
    \begin{equation*}\label{Eqn2-p}
        \sum_{i,r}\ell_{ir}\gamma_i=-\sum_{i,r}\ell_{ir}\alpha_r.
    \end{equation*}
    Next we use use the fact that 
$F_{2^e}\cap F_p=\Q$.  We may argue exactly as in
    Lemma~\ref{Lemma-pgst2Z} to deduce that the common  value of \eqref{Eqn2-p} is zero.
    Thus, if we set 
    $a_r=\sum_i\ell_{ir}$, we have
\begin{equation*}\label{alphap}
   \sum_{r=1}^{p-1}a_r\alpha_r=\sum_{r=1}^{\frac{p-1}{2}} (a_r-a_{p-r})\alpha_r=0. 
\end{equation*}
We saw in the proof of Lemma~\ref{Lemma-pgstPZ} that $\{\alpha_1, \dots, \alpha_{\frac{p-1}{2}}\}$ is a basis of $F_p$.  Hence $a_r=a_{p-r}$, for $r=1,\dots,\frac{p-1}{2}$, and
$\sum_{\text{$r$ even}} a_r = \sum_{\text{$r$ odd}} a_r$. Taking into account Condition~(b)(ii) of Theorem~\ref{PGSTcriterion}, we
obtain $\sum_{\text{$r$ even}} a_r=0$.
Therefore, by Theorem~\ref{PGSTcriterion},
we have pretty good state transfer from $(1,1)$ to $(1,p-1)$.
\ignore{
    We shall apply Theorem~\ref{PGSTcriterion}.
    The condition on strong cospectrality is satisfied, by Corollary~\ref{Lemma-scprods}.
    The eigenvalues of $P_{p-1}$ are  $\alpha_r:=2\cos(\frac{r\pi}{2p})$, $r=1$,\dots,$p-1$ and, to match our notation with Lemma~\ref{Lemma-pgst2Z}, we shall denote the eigenvalues 
    of $P_{2^e-1}$ by $\om^i+\om^{-i}$, $i=1$,\dots,$2^e-1$, where $\om$
    is a primitive $2^{e+1}$-th root of unity. The set $\Phi^-_{(1,1), (1,p-1)}$ consists of the eigenvalues
    $(\om_i+\om^{-i})+\alpha_r$ where
    $r$ is even.
    Suppose equation (b)(i)
    of Theorem~\ref{PGSTcriterion} holds.
    Thus for some integers $\ell_{ir}$, we have
    \begin{equation}
        \sum_{i,r}\ell_{i,r}((\om^i+\om^{-i})+\alpha_r)=0
    \end{equation}
    Then, as in Lemma~\ref{Lemma-pgst2Z} we
    move the eigenvalues of $P_{2^e-1}$
    to the left side and those of
    $P_{p-1}$ to the right side, 
    resulting in the equation
    \begin{equation}\label{Eqn2-p}
        \sum_{i,r}\ell_{ir}(\om^i+\om^{-i})=-\sum_{i,r}\ell_{ir}\alpha_r.
    \end{equation}
Next we use use the fact that 
$F_{2^e}\cap F_p=\Q$.  We may argue exactly as in
    Lemma~\ref{Lemma-pgst2Z} to deduce that the common  value of \eqref{Eqn2-p} is zero.
    Thus, if we set 
    $a_r=\sum_i\ell_{ir}$, we have
\begin{equation}\label{alphap}
   \sum_{r=1}^{p-1}a_r\alpha_r=0. 
\end{equation}
Let $\zeta=e^{\frac{i\pi}{p}}$, so that
$\alpha_r=\zeta^r+\zeta^{-r}=\zeta^r-\zeta^{p-r}$
and \eqref{alphap} becomes
\begin{equation}
     \sum_{r=1}^{p-1}(a_r-a_{p-r})\zeta^r=0. 
\end{equation}
As $\zeta$, $\zeta^2$,\dots,
$\zeta^{p-1}$ are linearly 
independent over $\Q$, it follows that $a_r=a_{p-r}$
for $r=1$,\dots, $p-1$.
Thus, if we assume condition (b)(ii) of Theorem~\ref{PGSTcriterion}
we have 
\begin{equation}
    0=\sum_{r=1}^{p-1}a_r=
2\sum_{\text{$r$ even}}a_r,
\end{equation}
so condition (b)(iii)
of Theorem~\ref{PGSTcriterion}
is false. Therefore, by
Theorem~\ref{PGSTcriterion},
we have pretty good state transfer from $(1,1)$ to $(1,p-1)$.
}
\end{proof}

\begin{lemma}\label{Lemma-pgst2e-2p} 
For prime $p\geq 3$ and $e\geq 2$, there is pretty good state transfer between $(1,1)$ and $(1,2p-1)$ in $P_{2^e-1}\square P_{2p-1}$.
\end{lemma}
\begin{proof}
The proof is very similar
to that of Lemma~\ref{Lemma-pgst2e-p}. We shall apply Theorem~\ref{PGSTcriterion}.
    The condition on strong cospectrality is satisfied, by Corollary~\ref{Lemma-scprods}.
    Recall the eigenvalues of $P_{2p-1}$ are  $\beta_r=2\cos(\frac{r\pi}{2p})$, $r=1$,\dots,$2p-1$.   Note that $\beta_p=0$.
    The set $\Phi^-_{(1,1), (1,2p-1)}$ consists of the eigenvalues
    $\gamma_i+\beta_r$ where $r$ is even.
    Suppose Condition~(b)(i) of Theorem~\ref{PGSTcriterion} holds.
    Thus for some integers $\ell_{ir}$, we have
    \begin{equation*}
        \sum_{i,r}\ell_{ir}(\gamma_i+\beta_r)=0
    \end{equation*}
    Then, as in Lemma~\ref{Lemma-pgst2Z} we
    move the eigenvalues of $P_{2^e-1}$
    to the left side and those of
    $P_{2p-1}$ to the right side, 
    resulting in the equation
    \begin{equation*}\label{Eqn2-2p}
        \sum_{i,r}\ell_{ir} \gamma_i =-\sum_{i,r}\ell_{ir}\beta_r.
    \end{equation*}
Next we use use the fact that 
$F_{2^e}\cap F_{2p}=\Q$.  We may argue exactly as in
    Lemma~\ref{Lemma-pgst2Z} to deduce that the common  value of \eqref{Eqn2-2p} is zero.
    Thus, if we set 
    $a_r=\sum_i\ell_{ir}$, we have 
\begin{equation*}\label{alpha2p}
   \sum_{r=1}^{2p-1}a_r\beta_r =\sum_{r=1}^{p-1} (a_r-a_{2p-r})\beta_r =0. 
\end{equation*}
Since $\beta_1$ has degree $p-1$, $\{\beta_1, \dots, \beta_{p-1}\}$ is a basis of $F_{2p}$ which implies
$a_r=a_{2p-r}$, for $r=1,\dots,2p-1$.  Thus
\begin{equation*}
\sum_{\text{$r$ even}} a_r = 2\left(a_2+a_4+\dots + a_{p-1}\right)
\end{equation*}
so Condition
(b)(iii) of Theorem~\ref{PGSTcriterion} can never hold. Therefore by Theorem~\ref{PGSTcriterion}, we have
pretty good state transfer between $(1,1)$ and $(1,2p-1)$.
\ignore{
The proof is very similar
to that of Lemma~\ref{Lemma-pgst2e-p}. We shall apply Theorem~\ref{PGSTcriterion}.
    The condition on strong cospectrality is satisfied, by Corollary~\ref{Lemma-scprods}.
    The eigenvalues of $P_{2p-1}$ are  $\beta_r:=2\cos(\frac{r\pi}{4p})$, $r=1$,\dots,$2p-1$ and, to match our notation with Lemma~\ref{Lemma-pgst2Z}, we shall denote the eigenvalues 
    of $P_{2^e-1}$ by $\om^i+\om^{-i}$, $i=1$,\dots,$2^e-1$, where $\om$
    is a primitive $2^{e+1}$-th root of unity. The set $\Phi^-_{(1,1), (1,2p-1)}$ consists of the eigenvalues
    $(\om_i+\om^{-i})+\beta_r$ where
    $r$ is even.
    Suppose equation (b)(i)
    of Theorem~\ref{PGSTcriterion} holds.
    Thus for some integers $\ell_{ir}$, we have
    \begin{equation}
        \sum_{i,r}\ell_{i,r}((\om^i+\om^{-i})+\beta_r)=0
    \end{equation}
    Then, as in Lemma~\ref{Lemma-pgst2Z} we
    move the eigenvalues of $P_{2^e-1}$
    to the left side and those of
    $P_{2p-1}$ to the right side, 
    resulting in the equation
    \begin{equation}\label{Eqn2-2p}
        \sum_{i,r}\ell_{ir}(\om^i+\om^{-i})=-\sum_{i,r}\ell_{ir}\beta_r.
    \end{equation}
Next we use use the fact that 
$F_{2^e}\cap F_{2p}=\Q$.  We may argue exactly as in
    Lemma~\ref{Lemma-pgst2Z} to deduce that the common  value of \eqref{Eqn2-2p} is zero.
    Thus, if we set 
    $a_r=\sum_i\ell_{ir}$, we have
\begin{equation}\label{alpha2p}
   \sum_{r=1}^{2p-1}a_r\beta_r=0. 
\end{equation}
Let $\xi=e^{\frac{i\pi}{2p}}$, so that
$\beta_r=\xi^r+\xi^{-r}=\xi^r-\xi^{2p-r}$
and \eqref{alpha2p} becomes
\begin{equation}
     \sum_{r=1}^{2p-1}(a_r-a_{2p-r})\xi^r=0. 
\end{equation}
Dividing by $\xi$, we have an integer polynomial equation in $\xi$ of degree at most $2(p-1)$. Since the minimal polynomial of $\xi$ is
$\sum_{i=0}^{p-1}(-1)^ix^{2i}$
it follows that the coefficient of
every odd power of $\xi$ in the
polynomial equation must be zero.
Hence $a_r=a_{2p-r}$ for every even
$r$ in the range $1\leq r\leq 2p-1$.
It follows that $\sum_{\text{$r$ even}}a_r$
must be an even number, so condition
(b)(iii) of Theorem~\ref{PGSTcriterion} can never hold. Therefore by Theorem~\ref{PGSTcriterion}, we have
pretty good state transfer between $(1,1)$ and $(1,2p-1)$.
}
\end{proof}

\subsection{Proof of Theorem~\ref{classification}}
With Lemma~\ref{Lemma-1factor} and the results in Sections~\ref{nopgstsection} and \ref{pgstsection} at our disposal we are now ready to prove Theorem~\ref{classification}, the classification of path products in which there is pretty good transfer among all corners.

    Suppose $X=P_{n_1}\square \dots \square P_{n_k}$, $k\geq 2$, has pretty good state transfer occurring between any two corners.
    By Lemma~\ref{Lemma-1factor}, $n_i+1$ is either a power of two, $p$ or $2p$ for some prime $p$, for $i=1, \dots, k$.  By Lemma~\ref{Lemma-gcd3-pgst}, we can assume that there are distinct primes $p_1, \dots, p_f, q_1, \dots, q_h$ such that 
    \begin{equation*}
        X=P_{p_1-1}\square \dots \square P_{p_f-1} \square P_{2q_1-1} \square \dots P_{2q_h-1}
    \end{equation*}
    where $f$ and $h$ are non-negative integers whose sum is at least two, 
    or 
    \begin{equation*}
        X=P_{2^e-1}\square P_{p_1-1}\square \dots \square P_{p_f-1} \square P_{2q_1-1} \square \dots P_{2q_h-1},
    \end{equation*}
    for $e\geq 2$ and $f+h\geq 1$.

    Suppose $f+h\geq 2$.  It follows from Lemmas~\ref{Lemma-p-q-3mod4} to \ref{Lemma-2p-q-all-3mod4} that $p_1,\dots, p_f \equiv 1 \pmod 8$ and $q_1, \dots, q_h \equiv 1 \pmod 4$.
    Thus, $X$ has the form of 
    parts (3) or (4) in the statement of Theorem~\ref{classification}.
    
    Conversely, suppose $X$ has
    the form of (3) or (4) in Theorem~\ref{classification}.
    We will show that $X$ has pretty good state transfer among all of its corners. By the discussion at the beginning of this section, it suffices to show that there is pretty good state transfer between $(1,1,\ldots,1)$ and an adjacent corner, which we can
    assume to be $(n_1,1,\ldots,1)$,
    where $n_1$ is one of the path
    lengths in (3) or (4).
    First by Lemma~\ref{Lemma-scprods}, we have strong cospectrality between these
    vertices. Then, depending
    on $n_1$, we can apply 
    Lemmas~\ref{Lemma-pgstPZ} to \ref{Lemma-pgst2Z}, taking
    $Z$ to be the product of the other factors in $X$ and $z$
    to be the vertex $(1,\ldots,1)$ of $Z$, to deduce pretty good state transfer. We must check that the
    hypotheses on the field intersections in Lemmas~\ref{Lemma-pgstPZ} to \ref{Lemma-pgst2Z} are satisfied. The eigenvalues 
    of a path $P_n$ lie in the intersection of the cyclotomic
    field of order $2(n+1)$ with the real field. 
    It follows that the field
    $F_{n_1+1}\cap F_Z$ in these
    Lemmas is a subfield of the threefold intersection  of the real numbers, a cyclotomic field of order $2(n_1+1)$ and a cyclotomic field of order $m$, where $\gcd(m,2(n_1+1))=4$. Hence $F_{n_1+1}\cap F_Z=\Q$.
    This completes the proof
    that the graphs $X$ in (3) and (4) have pretty good state transfer among all corners.
    
    When $f+h=1$,  Lemmas~\ref{Lemma-scprods}, \ref{Lemma-pgst2Z} to \ref{Lemma-pgst2e-2p}, show that pretty good state transfer occurs among all four corners of $P_{2^e-1}\square P_{p-1}$ and $P_{2^e-1}\square P_{2p-1}$ for $e\geq 2$ and prime $p\geq 3$.
 \qed
\section{No Laplacian Pretty good state transfer among corner vertices}
\label{Lpgstsection}

For the Heisenberg Hamiltonian, the continuous-time quantum walk on $X$ is given by $\exp(-itL_X)$, where $L_X$ is the Laplacian matrix of $X$  \cite[IV.E]{K}.
 We shall see in this section that, contrary to Theorem~\ref{classification}, there is no cartesian product of two or more paths with Laplacian pretty good state transfer between any two corners. 

 Given two simple finite graphs $X$ and $Y$ on $n$ and $m$ vertices, respectively, the Laplacian matrix of their cartesian product is 
 \begin{equation*}
     L_{X\square Y} = L_X \otimes I_m + I_n \otimes L_Y.
 \end{equation*}
 The transition matrix of $X\square Y$ based on the Heisenberg Hamiltonian is
 \begin{equation*}
     \exp(-itL_X) \otimes \exp(-itL_Y).
 \end{equation*}
 Hence Lemmas~\ref{Lemma-1factor} and \ref{Lemma-prod-idem} apply to the Laplacian matrix of $X\square Y$.

Laplacian pretty good state transfer occurs between extremal vertices of $P_n$ if and only if $n$ is a power of $2$ \cite{BCGS}. 
By Lemma~\ref{Lemma-1factor}, if $P_{n_1}\square\cdots \square P_{n_k}$ has Laplacian pretty good state transfer between any two corners then each $n_i$ is a power of $2$.
The Laplacian matrix of a path $P_{2^e}$ has eigenvalues $0$, and
$2+2\cos \frac{\pi r}{2^e}$, for $r=1,\dots,2^e-1$.  
The idempotent projector of the $0$-eigenspace is $\frac{1}{n}J_n$, where $J_n$ denotes the $n\times n$ matrix of all ones.
The extremal vertices of $P_{2^e}$ have full eigenvalue support with 
\begin{equation*}
    \Phi^-_{1,2^e} = \left\{2+2\cos \frac{\pi r}{2^e} \ : \ \text{$r$ odd}\right\}.
\end{equation*}

\begin{lemma}\label{Lemma-Lap}
Suppose $f\geq e\geq 1$.  The vertices $(1,1)$ and $(2^e,1)$ in $P_{2^e} \square P_{2^f}$ are not strongly cospectral.
\end{lemma}
\begin{proof}
Let $E_r$ be the idempotent projector of the $\left(2+2\cos \frac{\pi r}{2^e}\right)$-eigenspace of $L(P_{2^e})$,
and let $F_s$ be the idempotent projector of the $\left(2+2\cos \frac{\pi s}{2^f}\right)$-eigenspace of $L(P_{2^f})$.

Let $W$ be the idempotent projector of the Laplacian matrix of $P_{2^e}\square P_{2^f}$ corresponding to the eigenvalue $\theta = 2+2\cos\frac{\pi}{2^e}$.
As 
\begin{equation*}
   \theta = \left(2+2\cos\frac{\pi}{2^e}\right)+ 0 = 0 + \left(2+2\cos\frac{2^{f-e} \pi }{2^f}\right),  
\end{equation*}both $\left(E_1\otimes \left(\frac{1}{2^f}J_{2^f} \right)\right) W$ and 
$\left(\left(\frac{1}{2^e}J_{2^e} \right) \otimes F_{2^{f-e}}\right)W$ are non-zero.  Since $E_1 e_1 = - E_1e_{2^e}$ and $\left(\frac{1}{2^e}\right)J_{2^e}e_1=+\left(\frac{1}{2^e}\right)J_{2^e}e_{2^e}$, it follows from Lemma~\ref{Lemma-prod-idem} that $We_{(1,1)} \neq \pm W e_{(2^e,1)}$.

\end{proof}

\begin{theorem}
There is no Laplacian pretty good state transfer among the corners of $P_{n_1}\square\cdots \square P_{n_k}$, for all $n_1,\dots, n_k \geq 2$.  
\end{theorem}
\begin{proof}
By Lemma~\ref{Lemma-1factor}, we need to consider only the case where $n_1, \dots, n_k$ are powers of two. 

First observe that $P_{n_1}\square\cdots \square P_{n_k}$ is isomorphic to $\left(P_{n_1}\square P_{n_2}\right) \square \left(P_{n_3}\square \dots \square P_{n_k}\right)$ with an isomorphism given by
\begin{equation*}
    (v_1,v_2,\dots,v_n) \mapsto \big((v_1,v_2), (v_3,\dots, v_k)\big).
\end{equation*}
Lemmas~\ref{Lemma-Lap} and \ref{Lemma-1factor} rule out Laplacian pretty good state transfer from $\left((1,1),(1,\dots,1)\right)$ to $\left((n_1,1),(1,\dots,1)\right)$ in $\left(P_{n_1}\square P_{n_2}\right) \square \left(P_{n_3}\square \dots \square P_{n_k}\right)$ when $n_1$ and $n_2$ are powers of two.
Hence there is no Laplacian pretty good state transfer from $(1, 1, \dots, 1)$ to $(n_1, 1, \dots, 1)$ in $P_{n_1}\square\cdots \square P_{n_k}$.\end{proof}

  \section*{Acknowledgements}
  This work was done during a visit by both authors to the University of Waterloo.
  We thank Chris Godsil for making this visit possible and for many interesting
  discussions. We also thank Pedro Vinícius Ferreira Baptista and Gabriel Coutinho
  of the Federal University of Minas Gerais for  some computer calculations of examples.


\end{document}